\documentclass[12pt,a4paper,oneside]{amsart}
\usepackage{color,amssymb,latexsym,amsfonts,textcomp,fancyhdr,calc,graphicx}
\usepackage{amsmath,amstext,amsthm,amssymb,amsxtra}
\usepackage{txfonts}
\usepackage[colorlinks=true]{hyperref}
\usepackage{txfonts,pxfonts}
\usepackage[T1]{fontenc}

\newtheorem{theorem}{Theorem}
\newtheorem{corollary}[theorem]{Corollary}
\newtheorem{lemma}{Lemma}

\theoremstyle{definition}

\newcommand{\beql}[1]{\begin{equation}\label{#1}}
\newcommand{\eeq}{\end{equation}}
\newcommand{\comment}[1]{}

\newcommand{\Abs}[1]{{\left|{#1}\right|}}

\newcommand{\Set}[1]{{\left\{{#1}\right\}}}

\newcommand{\RR}{{\mathbb R}}

\newcommand{\CC}{{\mathbb C}}
\newcommand{\ZZ}{{\mathbb Z}}

\newcommand{\one}{{\bf 1}}

\newcommand{\supp}{{\rm supp\,}}

\newcommand{\ft}[1]{\widehat{#1}}

\newcounter{rem}
\newcounter{step}
\newcounter{mysec}
\newcounter{mysubsec}[mysec]
\newcounter{othm}
\def\theothm{\Alph{othm}}

\begin{document}

\title[No measurable Steinhaus sets]{Measurable Steinhaus sets do not exist for finite sets or the integers in the plane}

\author{Mihail N. Kolountzakis and Michael Papadimitrakis}
\address{Department of Mathematics and Applied Mathematics, University of Crete, Voutes Campus, GR-700 13, Heraklion, Crete, Greece}
\email{kolount@gmail.com}
\email{papadim@math.uoc.gr}

\begin{abstract}
A Steinhaus set $S \subseteq \RR^d$ for a set $A \subseteq \RR^d$ is a set such that $S$ has exactly one point in common with
$\tau A$, for every rigid motion $\tau$ of $\RR^d$.
We show here that if $A$ is a finite set of at least two points then there is no such set $S$ which is Lebesgue measurable.

An old result of Komj\'ath says that there exists a Steinhaus set for $A = \ZZ\times\Set{0}$ in $\RR^2$.
We also show here that such a set cannot be Lebesgue measurable.
\end{abstract}

\maketitle

\section{Introduction}

A classical question  of Steinhaus \cite{sierpinski1958probleme,moser1981research} was if there exist two proper subsets $A, B \subseteq \RR^2$
such that any rigid motion of $A$ intersects $B$ at exactly one point.
This was answered in the affirmative by Sierpinski \cite{sierpinski1958probleme}.

Let us rephrase the question. If $\rho \in O(2)$ denotes an arbitrary orthogonal transformation and $x \in \RR^2$ then
a general rigid motion of $A$ is $x+\rho A$ and we are asking that
$$
\sum_{b \in B} \one_{x+\rho A}(b) = 1,\ \ \ \text{for all $\rho\in O(2)$ and all $x \in \RR^2$.}
$$
We can rewrite this (interpreting the infinite sum in the only possible way) as
\beql{general-tiling}
\sum_{b \in B} \one_{\rho A}(x-b) = 1,\ \ \ \text{for all $\rho \in O(2)$ and all $x \in \RR^2$}.
\eeq
In other words we are asking that for any $\rho \in O(2)$ the $B$-translates of $\rho A$ form a partition of $\RR^2$.
This is, in other words, a {\em tiling} condition: all sets $\rho A$ tile space when translated by the same set $B$,
a situation which we may denote by
$$
\rho A \oplus B = \RR^2.
$$

\noindent{\bf Measurablity in the Steinhaus question}.
Condition \eqref{general-tiling} may be altered by
\begin{itemize}
\item [(a)] demanding that the sets $A$ and $B$ are Lebesgue measurable (usually one
of them is countable and hence automatically measurable and of measure 0) and
\item[(b)] only asking that \eqref{general-tiling} holds for almost all $x \in \RR^2$.
\end{itemize}
We call this a {\em measurable} Steinhaus problem.

A further question of Steinhaus was if there is a set $A \subseteq \RR^2$ which intersects every rigid motion of $\ZZ^2$
at exactly one point. Let us call this a {\em lattice Steinhaus set}.
Thus set $B$ is given as $\ZZ^2$ and one asks for a set $A$ which satisfies \eqref{general-tiling}.

Komj\'ath \cite{komjath1992lattice} answered the Steinhaus question in the affirmative when $B=\ZZ\times\Set{0}$
showing that there are Steinhaus sets for this $B$.
Sierpinski \cite{sierpinski1958probleme} showed that a bounded set $A$ which is either closed or open cannot have the
lattice Steinhaus property (that is, intersect all rigid motions of $\ZZ^2$ at exactly one point).
Croft \cite{croft1982three} and Beck \cite{beck1989lattice} showed that no bounded and measurable set $A$ can
have the lattice Steinhaus property (but see also \cite{mallinikova1995}).
Kolountzakis \cite{kolountzakis1996problem,kolountzakis1996new} and Kolountzakis and Wolff \cite{kolountzakis1999steinhaus}
proved that any measurable set in the plane that has the measurable Steinhaus property must necessarily
have very slow decay at infinity (it is easy to see that any such set must have measure 1). In \cite{kolountzakis1999steinhaus}
it was also shown that there can be no measurable Steinhaus sets in dimension $d\ge 3$ (with the obvious definition)
a fact that was also shown later by Kolountzakis and Papadimitrakis \cite{kolountzakis2002steinhaus} by a very different method.
See also \cite{chan2007steinhaus,mauldin2003comments,ciucu1996remark,srivastava2005steinhaus}. Kolountzakis \cite{kolountzakis1997multi} looks at the case where
the linear transformation $\rho$ is only required to take on finitely many values.
In a major result Jackson and Mauldin \cite{jackson2002sets,jackson2002lattice} proved the existence of lattice Steinhaus sets in the plane
(not necessarily measurable). Their method does not extend to higher dimension $d \ge 3$. See also \cite{mauldin2001some,jackson2003survey}.

In this paper we deal with the measurable Steinhaus problem for the case of finite sets $B$.
So far there have been negative results without assuming measurability and only for small cardinalities of the set $B$.
Gao, Miller and Weiss \cite{gao2006steinhaus} prove that for a large class of 4-point sets $B$ in the plane there is
no corresponding Steinhaus set $A$.
Xuan \cite{xuan2012steinhaus} proves this for all 4-point sets $B$.

\noindent{\bf New results.}
Our first result is that in the measurable case this is true for all finite sets.
\begin{theorem}\label{th:main}
If $B \subseteq \RR^d$ is a finite set, $\Abs{B}>1$, then there is no measurable Steinhaus set for $B$.
That is, there is no Lebesgue measurable set $A \subseteq \RR^d$ such that for all $\rho \in O(d)$
$$
\sum_{b \in B} \one_{\rho A}(x-b) = 1,\ \ \ \text{for almost all $x \in \RR^d$}.
$$
\end{theorem}
Since \cite{gao2006steinhaus,xuan2012steinhaus} do not deal with necessarily measurable sets $A$ their results
are not directly comparable to our result.

The method of proof follows a Fourier Analytic approach to tiling \cite{kolountzakis2004study}
accompanied by an interesting result about the zero set of multivariate trigonometric polynomials (Theorem \ref{th:circle}).
The proof is given in \S\ref{sec:proof}.

We also prove, without Fourier Analysis this time, that a set such as Komjath \cite{komjath1992lattice} constructed cannot
be Lebesgue measurable.

\begin{theorem}\label{th:z}
There is no Lebesgue measurable set $S \subseteq \RR^2$ such that for all $\rho \in O(2)$
$$
\sum_{n \in \ZZ\times\Set{0}} \one_S(x-\rho n) = 1,\ \ \ \text{for almost all $x \in \RR^2$}.
$$
\end{theorem}

A Fourier Analytic approach to Theorem \ref{th:z}, along the lines of the proof of Theorem \ref{th:main},
 is made difficult by the fact that such a set $S$, if it exists, must necessarily have infinite measure.

Theorem \ref{th:z} turns out to be a consequence of the following fact: there is no Lebesgue measurable set
in the plane which intersects almost every line at the same measure (Theorem \ref{th:harmonic} and Theorem \ref{th:harmonic-bounds}).
The proof appears in \S\ref{sec:z-proof}.

\section{No measurable Steinhaus sets for finite sets}\label{sec:proof}

Our first result is very similar (see e.g.\ \cite{kolountzakis2004study}) to the
corresponding result which covers tiling Euclidean space by a set of finite measure.
The proof is not exactly the same, owing to the fact that our set $S$ must necessarily
have infinite measure.

\noindent{\bf Notation:} For any discrete set of points $B$ we denote by $\delta_B$ the measure which
assigns a unit point mass to each point of $B$.
 
\begin{theorem}\label{th:supp}
Suppose $S \subseteq \RR^d$ is a measurable set and $B \subseteq \RR^d$ is a finite set
such that $\one_S*\delta_B=1$ almost everywhere.
Then for the tempered distribution $\ft{\one_S}$ and the trigonometric polynomial
$$
\ft{\delta_B}(x) = \sum_{b \in B} e^{-2\pi i b\cdot x}
$$
we have
\beql{supp}
\supp{\ft{\one_S}} \subseteq \Set{0} \cup \Set{\ft{\delta_B}=0}.
\eeq
\end{theorem}

\begin{proof}
Let $K = \Set{0}\cup \Set{\ft{\delta_B}=0}$, which is a closed set. Inclusion \eqref{supp}
means (by the definition of the support of a tempered distribution) that
$\ft{\one_S}(\psi) = 0$ for all smooth
$\psi$ supported in $K^c$.
For such a $\psi$ we have
\beql{conv}
\left( \widetilde{\ft{\delta_B}} \, \psi \right)^\wedge (\lambda) = (\widetilde{\delta_B}*\ft{\psi})(\lambda),
\ \ \ \ (\lambda \in \RR^d)
\eeq
where we use the notation $\widetilde{f}(x) = \overline{f(-x)}$, extended by duality to tempered distributions.

We must show $\ft{\one_S}(\psi) = 0$.
We have
$$
\ft{\one_S}(\psi) = \ft{\one_S}\left(
    \widetilde{\ft{\delta_B}} \cdot {\psi \over \widetilde{\ft{\delta_B}} } \right).
$$
Notice that $\widetilde{\ft{\delta_B}}=\ft{\delta_B}$ (since $\delta_B$ is real),
so the quotient $\phi = \psi / \widetilde{\ft{\delta_B}}$ is a $C^\infty_0(K^c)$
function.
We have
\begin{eqnarray*}
\ft{\one_S}(\psi)
 &=& \ft{\one_S} \left( \widetilde{\ft{\delta_B}} \,\phi \right) \\
 &=& \one_S \left( \left( \widetilde{\ft{\delta_B}} \phi \right)^\wedge \right)
\ \ \text{(by the definition of the Fourier Transform for distributions)}\\
 &=& \int \one_S(\lambda) \left( \widetilde{\ft{\delta_B}} \phi \right)^\wedge (\lambda) \,d\lambda\\
 &=& \int \one_S(\lambda) \left(\widetilde{\delta_B}*\ft\phi\right)(\lambda) \,d\lambda
\ \ \text{(by \eqref{conv})}\\
 &=& \int \one_S(\lambda) \int \ft\phi(x)\,d\widetilde{\delta_B}(\lambda-x) \,d\lambda\\
 &=& \int \int \one_S(\lambda) \,d\delta_B(x-\lambda) \ \ft\phi(x)\,dx\\
 &=& \int \ft\phi(x)\,dx \ \ \text{(since $\one_S*\delta_B=1$)}\\
 &=& \phi(0)\\
 &=& 0
\ \ \mbox{(as $0 \notin \supp\phi$)}.
\end{eqnarray*}

\end{proof}

\begin{theorem}\label{th:vanishing}
If $S \subseteq \RR^d$ is a measurable Steinhaus set for the finite set $B \subseteq \RR^d$, with cardinality $\Abs{B}>1$, then
the trigonometric polynomial
$\ft{\delta_B}$ vanishes on a sphere in $\RR^d$ centered at the origin.
\end{theorem}

\begin{proof}
Every isometry of $\RR^d$ can be written as $\tau = x+\rho\cdot$, for some $x \in \RR^d$ and some $\rho\in O(d)$,
so for all $x \in \RR^d, \rho \in O(d)$ we have
$$
\sum_{b \in B} \one_S(x+\rho b) = 1.
$$
This is in turn equivalent to
$$
\sum_{b \in B} \one_{\rho S}(x-b) = 1,\ \ \ \text{for all $x \in \RR^d, \rho \in O(d)$}.
$$
This can be written as
\beql{conv1}
\one_{\rho S}*\delta_B(x) = 1,\ \ \ \text{for all $x \in \RR^d$}.
\eeq
We only need to assume that \eqref{conv1} holds for almost all $x \in \RR^d$.

Since $\Abs{B}>1$ the set $\rho S$ cannot be almost all of $\RR^d$
so the tempered distribution $\ft{\one_{\rho S}}$ cannot have $\supp{\ft{\one_{\rho S}}} = \Set{0}$.
Suppose then that $0 \neq v \in \supp{\ft{\one_{\rho S}}}$. By Theorem \ref{th:supp} it follows that $\ft{\delta_B}(v) = 0$.
Varying $\rho$ we conclude that $\ft{\delta_B}$ vanishes on the entire sphere with radius $\Abs{v}$ centered at the origin.
\end{proof}

\begin{theorem}\label{th:circle}
If $\psi(x)$ is a trigonometric polynomial on $\RR^2$ which vanishes on a circle then it is identically zero.
\end{theorem}

\begin{proof}
Suppose, after appropriately modulating the coefficients and scaling the frequencies $b \in B$ (a finite set in the plane), that
$$
\psi(x, y) = \sum_{b \in B} c_b e^{2\pi i b\cdot(x,y)}
$$
vanishes on the unit circle centered at the origin.
View the frequencies $b \in B$ as complex numbers $b = b_1+i b_2$ and write also $z = x-iy$,
so that for $\Abs{z}=1$ the function
$$
f(z) = \sum_{b \in B} c_b e^{2\pi i \Re(bz)} = \sum_{b \in B} c_b e^{\pi i(bz+\overline{b}/z)}
$$
vanishes identically. The function $f(z)$ is defined and analytic on $\CC\setminus\Set{0}$, hence, by
analytic continuation, vanishes identically for all $z\neq 0$.

Let $b_0 \in B$ have maximal modulus. We may suppose that $b_0$ is unique by appropriately translating the set $B$,
an operation which does not change the zeros of $\psi(x,y)$.
Write $z=\theta t$, where $\Abs{\theta}=1$ and $t \to + \infty$ and select $\theta$ of modulus 1 to be such that
$$
\pi i b_0 \theta = \pi \Abs{b_0}.
$$
For all $b \in B \setminus \Set{b_0}$ we have $\Re(\pi i b \theta) < \pi \Abs{b_0}$.
We have, as $t \to +\infty$,
$$
0 = f(\theta t) = c_{b_0} e^{\pi \Abs{b_0} t + O(1/t)} + \sum_{b\in B \setminus\Set{b_0}} c_b e^{\pi i b \theta t + O(1/t)}
$$
and the modulus of every term in the sum above is $\Abs{c_b} e^{\Re(\pi i b \theta)t + O(1/t)}$.
Since the first term in the sum above is dominant we have reached a contradiction.
\end{proof}

\begin{corollary}\label{cor:final}
If $\psi(x)$ is a trigonometric polynomial on $\RR^d$, $d>1$, which vanishes on a sphere in $\RR^d$
then it is identically zero.
\end{corollary}

\begin{proof}
Restricting $\psi$ on any 2-plane in $\RR^d$ leaves us with a two-variable trigonometric polynomial that vanishes on a circle.
It follows from Theorem \ref{th:circle} that the restriction of $\psi$ on any 2-plane is identically 0, hence $\psi$ is
identically zero on $\RR^d$.
\end{proof}

From Theorem \ref{th:vanishing} and Corollary \ref{cor:final} we would have a contradiction if $S$ were a Steinhaus set
for $B$. The proof of Theorem \ref{th:main} is complete.

\section{No measurable Steinhaus sets for the integers in the plane}\label{sec:z-proof}

In this section we prove Theorem \ref{th:z}.
We assume that $S$ is a measurable Steinhaus set for $\ZZ\times\Set{0}$ in $\RR^2$.
First we show that such a set $S$ must have a very strong property.

\begin{lemma}\label{lm:lines}
Suppose $S \subseteq \RR^2$ is Lebesgue measurable and for all $\rho \in O(2)$ we have
\beql{tiling-with-z}
\sum_{n \in \ZZ\times\Set{0}} \one_S(x-\rho n) = 1,\ \ \ \ \text{for almost all $x \in \RR^2$.}
\eeq
Then for every $\rho\in O(2)$ and for almost all lines $L$ in the plane parallel to $\rho(\RR\times\Set{0})$ we have
\beql{intersection}
\Abs{S \cap L} = 1,
\eeq
that is $S$ has measure 1 when restriced on almost any straight line $L$ of any slope.

We also have that for almost all points $(x_0, y_0) \in \RR^2$ almost all lines $L$ through $(x_0, y_0)$
have property \eqref{intersection}.
\end{lemma}

\begin{proof}
The function on the left of \eqref{tiling-with-z}
is equal to 1 almost everywhere so for almost all lines $L$ parallel to $\rho(\RR\times\Set{0})$
the function is measurable and equal to 1 almost everywhere as a function of one variable.
For any such line $L$ integrate \eqref{tiling-with-z} for $x$ in a unit line segment on $L$ to get that $\Abs{S \cap L} = 1$.

Ignoring lines parallel to the $x$-axis we can uniquely parametrize all lines by the pair $(x, \theta)$, with $\theta \in (0,\pi)$ being
the angle the line forms with the $x$-axis and $x \in \RR$ being the point of intersection of the line with the $x$-axis.
So far we have shown that for all $\theta$ and almost all $x$ the corresponding line $L$ satisfies \eqref{intersection}.

We can also redundantly parametrize all lines not parallel to the $x$-axis
by the triple $(x_0, y_0, \theta)$, this line being the one through point $(x_0, y_0)$
forming angle $\theta$ with the $x$-axis and let
$$
\phi(x_0, y_0, \theta) = (x, \theta)
$$
be the continuous function that translates from one parametrization to the other.

Let $E \subseteq \RR\times (0,\pi)$ be the exceptional set for property \eqref{intersection} in the first parametrization
$$
E = \Set{(x, \theta): \text{the line $(x, \theta)$ fails \eqref{intersection}}}.
$$
Then $E$ has measure 0 and therefore for any $\epsilon>0$
we can find an open cover $E \subseteq \bigcup_n I_n$ of intervals (rectangles) with
$\sum_n\Abs{I_n} < \epsilon$.
Fix any $R>0$ and let $Q_R = [-R,R]\times[-R,R]\times(0,\pi)$.
By the simple geometry we deduce that for any interval $I$ we have
$$
\Abs{\phi^{-1}(I) \cap Q_R} = O(R)\Abs{I}.
$$
In fact, if $I=J\times K$, then, fixing any $\theta\in K$, the set of $(x_0,y_0)$ such that $\phi(x_0,y_0, \theta)\in J\times K$ has measure $O(R)\Abs{J}$. 
Therefore $\Abs{\phi^{-1}(I) \cap Q_R} =O(R)\Abs{J}\Abs{K}=O(R)\Abs{I}$. Now, since $\epsilon>0$ is arbitrary, 
we get that $\Abs{\phi^{-1}(E) \cap Q_R} = 0$, and, since $R>0$ is arbitrary, $\Abs{\phi^{-1}(E)} = 0$.
We have proved that for almost all triples $(x_0, y_0, \theta) \in \RR^2 \times (0, \pi)$ the line
indexed by $(x_0, y_0, \theta)$ has property \eqref{intersection}. Applying Fubini's theorem we conclude that
for almost all $(x_0, y_0) \in \RR^2$ for almost all $\theta \in (0, \pi)$ the line $L$ indexed by $(x_0, y_0, \theta)$
has property \eqref{intersection}.
\end{proof}

We now show that no Lebesgue measurable set $S$ can have the property described in Lemma \ref{lm:lines}.
Theorem \ref{th:harmonic} completes the proof of Theorem \ref{th:z}.

\begin{theorem}\label{th:harmonic}
There is no Lebesgue measurable $S \subseteq \RR^2$ which
for almost all $(x_0, y_0) \in \RR^2$
intersects
almost all lines through $(x_0, y_0)$ in measure 1.
\end{theorem}

\begin{proof}
Define the function $f:\RR^3\to\RR^+$ by
$$
f(z) = \int_{\RR^2} \one_S(w) \Abs{z-w}^{-1} \,dw.
$$

\noindent{\bf Claim:} $f(z) = \pi$ for almost all $z \in \RR^2$.

Indeed with $z \in \RR^2$ we have
\begin{align*}
f(z) &= \int_{\RR^2} \one_S(w) \frac{dw}{\Abs{z-w}}\\
 &= \int_{\RR^2} \one_S(z+w) \frac{dw}{\Abs{w}}\ \ \ \ \text{(change of variable)}\\
 &= \int_{[0,\pi]} \int_{\RR} \one_S(z+r(\cos\theta, \sin\theta)) \, dr \, d\theta\ \ \ \ \text{(polar coordinates)}\\
 &= \int_{[0,\pi]} 1 \,d\theta\ \ \ \ \text{(the integral is 1 for almost all $z$ and $\theta$)}\\
 &= \pi.
\end{align*}

\noindent{\bf Claim:} The function $f(z)$ is continuous in $\RR^3$.

To prove the claim we consider the orthogonal projection $z_0'$ of $z_0 \in \RR^3$ onto the plane $\RR^2$ of the first two coordinates and the disk
$D$ in $\RR^2$ centered at $z_0'$ with radius 1. Then
$$
f(z) =  \int_{D} \one_S(w) \frac{dw}{\Abs{z-w}} + \int_{D^c} \one_S(w) \frac{dw}{\Abs{z-w}}=g(z)+h(z).
$$

Now, if $z \to z_0$, then $g(z) \to g(z_0)$.

To show this we let $\delta=\Abs{z-z_0}$ and we write
$$
\Abs{g(z)-g(z_0)} \le \int_{D} \Abs{\frac 1{\Abs{z-w}}-\frac 1{\Abs{z_0-w}}}dw = \int_{D(2\delta)} \cdots + \int_{D\setminus D(2\delta)}\cdots
$$
where $D(2\delta)$ is the disk centered at $z_0'$ with radius $2\delta$.

Now
$$\int_{D(2\delta)} \cdots = \int_{D(2\delta)} \frac 1{\Abs{z-w}} dw + \int_{D(2\delta)} \frac 1{\Abs{z_0-w}} dw \le 2\int_{|u|\le 3\delta}\frac 1{\Abs{u}} du 
= O(\delta).
$$
Also
$$
\int_{D\setminus D(2\delta)}\cdots \le \int_{D\setminus D(2\delta)}\frac{\delta}{\Abs{z-w}\Abs{z_0-w}} dw \le \int_{\delta \le |u|\le 1}\frac{\delta}{\Abs{u}^2} du = O(\delta\log\frac 1{\delta}).
$$
Hence $g(z) \to g(z_0)$.

To prove that $h(z) \to h(z_0)$ when $z \to z_0$ we may consider a good point
$z_0''\in \RR^2$ such that $\Abs{z_0''-z_0'} < 1/2$ and we may also assume that 
$\Abs{z-z_0} < 1/2$. 
(We call a point of $\RR^2$ {\em good} if almost all lines through the point intersect $S$ at measure 1.
Almost every point of $\RR^2$ is good by Lemma \ref{lm:lines}.)
This implies that 
$$
\frac 1{\Abs{z-w}} \le \frac 3{\Abs{z_0''-w}}
$$
for all $w\in D^c$. Since
$$
\int_{D^c} \one_S(w) \frac{dw}{\Abs{z_0''-w}} \le \int_{\RR^2} \one_S(w) \frac{dw}{\Abs{z_0''-w}} = \pi,
$$
by dominated convergence we get $h(z) \to h(z_0)$ and the claim has been proved.

As a result we now have that $f(z) = \pi$ for all $z \in \RR^2$.

\noindent{\bf Claim:} The function $f(z)$ is harmonic in the upper half-space
$$
H=\Set{(x_1, x_2, x_3) \in \RR^3:\ x_3>0}.
$$

It is enough to show that $f(z)$ satisfies the mean value property in $H$.

Suppose $z_0 \in H$ and let $\Sigma$ be the surface of a sphere centered at $z_0$ and contained in $H$.
Let also $\sigma$ denote the surface measure of $\Sigma$. We have
\begin{align*}
\frac{1}{\sigma(\Sigma)} \int_\Sigma f(z) \, d\sigma(z)
 &= \frac{1}{\sigma(\Sigma)} \int_\Sigma \int_{\RR^2} \one_S(w) \frac{dw}{\Abs{z-w}} \,d\sigma(z)\\
 &= \int_{\RR^2} \one_S(w) \frac{1}{\sigma(\Sigma)} \int_\Sigma \frac{d\sigma(z)}{\Abs{z-w}} \,dw\\
 &= \int_{\RR^2} \one_S(w) \frac{dw}{\Abs{z_0-w}}\ \ \ \ \text{(as $1/\Abs{\cdot-w}$ is harmonic in $\overline\Sigma$)}\\
 &= f(z_0),
\end{align*}
so that $f(z)$ has the mean value property at $z_0$.

Considering the orthogonal projection $z'$ of $z \in \RR^3$ onto the plane $\RR^2$ and using $f(z') = \pi$ we get
$$0 \le f(z) \le \pi$$
for all $z\in\RR^3$.


Any function which is harmonic in the open upper half-space, bounded and continuous in the closed upper half-space is
the Poisson integral of its restriction on the plane (see e.g.\ \cite[Theorem 7.5]{axler2001harmonic}). Since $f$
is constant on the plane we deduce that $f \equiv \pi$ everywhere in the upper half-space. But this is a contradiction
as clearly $\lim_{z \to +\infty} f(x, y, z) = 0$ for all $(x, y) \in \RR^2$.

\end{proof}

\noindent{\bf Remark.}
It is not hard to see that with the same proof as Theorem \ref{th:harmonic}
we can prove the following stronger result.
\begin{theorem}\label{th:harmonic-bounds}
There is no Lebesgue measurable set $S \subseteq \RR^2$ which intersects almost all lines of the plane
at a measure bounded below by $m>0$ and above by $M<\infty$.
\end{theorem}

\bibliographystyle{abbrv}
\bibliography{steinhaus}

\end{document}